\documentclass[12pt, a4]{amsart}
\usepackage{amsgen,amsmath,amstext,amsbsy,amsopn,amsfonts,amssymb}
\usepackage{amsmath,amssymb,amsfonts,amsthm,enumerate}
\usepackage{mathrsfs}
\usepackage{amsthm}
\usepackage[alphabetic]{amsrefs}

\newtheorem{theorem}{Theorem}[section]
\newtheorem{lemma}[theorem]{Lemma}

\newtheorem{proposition}[theorem]{Proposition}
\newtheorem{obs}[theorem]{Observation} \newtheorem{defi}[theorem]{Definition}

\newenvironment{definition}{\begin{defi}\rm}{\end{defi}}
\newtheorem{exa}[theorem]{Example}

\newtheorem{rem}[theorem]{Remark}
\newenvironment{remark}{\begin{rem}\rm}{\end{rem}}
\newtheorem{rems}[theorem]{Remarks}

\newtheorem{ack}[theorem]{Acknowlegment}

\renewcommand{\emph}[1]{{\bf #1}}
\newtheorem{thmx}{Theorem}
 
\newtheorem{corx}[thmx]{Corollary}

\def\H{\mathcal H}

\def\K{\mathcal K}

\def\B{\mathcal B}

\def\ZZ{{\mathbf Z}}
\def\CC{{\mathbf C}}
\def\RRR{{\mathbf R}}

\def\RR+{{\mathbf R}^*}

\def\kk{k}

\def\Un{\mathbf 1}

\def\Q_p{{\mathbf Q}_p}

\def\Proj{\rm Proj}

\def\Ga{\Gamma}
\def\ga{\gamma}

\def\la{\lambda}
\def\vfi{\varphi}

\def\Tr{{\rm Tr}}
\def\Char{{\rm Char}}

\def\Sub{{\rm Sub}}

\def\tout{\quad\text{for all}\quad}

\def\Un{\mathbf 1}
\begin{document}
\date{\today}

\title[Characters of algebraic groups]{Character rigidity  of simple algebraic groups}

\address{Bachir Bekka \\ Univ Rennes \\ CNRS, IRMAR--UMR 6625\\
Campus Beaulieu\\ F-35042  Rennes Cedex\\
 France}
\email{bachir.bekka@univ-rennes1.fr}
\author{Bachir Bekka}

\thanks{The author acknowledges the support  by the ANR (French Agence Nationale de la Recherche)
through the projects Labex Lebesgue (ANR-11-LABX-0020-01) and GAMME (ANR-14-CE25-0004)}
\begin{abstract}
We prove the following extension of Tits'  simpli\-city theorem.
Let  $\kk$ be an infinite field, $G$ an algebraic group defined 
and quasi-simple over $\kk,$ and $G(\kk)$  the group of $\kk$-rational points
of $G.$ Let $G(\kk)^+$ be  the subgroup of $G(\kk)$ generated by the unipotent radicals
of parabolic subgroups of $G$ defined over $\kk$
and denote by $PG(\kk)^+$ the quotient of $G(\kk)^+$
by its center.

Then  every normalized function of positive type on $PG(k)^+$ which is constant on conjugacy classes  is a convex combination  of $\Un_{PG(\kk)^+}$
and $\delta_e.$
As corollary,  we obtain that, when $\kk$ is countable, the only ergodic  IRS's (invariant random subgroups) of $PG(\kk)^+$ are $\delta_{PG(\kk)^+}$ and $\delta_{\{e\}}.$
A further consequence is that, when $\kk$ is a global field
and $G$ is $\kk$-isotropic and has trivial center, every  measure preserving  ergodic action  of $G(\kk)$ on a probability space either factorizes through 
the abelianization $G(\kk)_{\rm ab}$ or is essentially free.
 \end{abstract}
\let\thefootnote\relax
\maketitle
\section{Introduction}
\label{S0}
Given a  locally compact group $G,$ recall that a continuous function   $\vfi \colon G \to \CC$ is  of positive type (or positive definite) if,  for all $g_1,\dots,g_n \in G$,
the matrix $(\varphi(g_j^{-1} g_i))_{1\leq i,j\leq n}$ is positive semi-definite.
As is well-known, such functions are exactly the diagonal matrix coefficients
of unitary representations of $G$ in  Hilbert spaces (see Section~\ref{SS:PreliminaryResult}). The set  $P(G)$
of functions $\vfi$ of positive type on $G$, normalized by the condition $\vfi(e)=1,$
is convex and  the  extreme points  in $P(G)$ are the  diagonal matrix coefficients,
given by unit vectors,  of  irreducible unitary representations of $G$.

Assume now that $G$ is a discrete group.  By results of Glimm \cite{Glimm} and Thoma \cite{Thoma5}, the classification of the unitary dual $\widehat{G}$, that is, the  set of irreducible unitary representations of $G$ up to unitary equivalence, is hopeless, unless $G$ is virtually abelian. However, the  set 
of characters or traces of $G$, that  we are going to define, seems to be more tractable.
\begin{definition}
\label{Def-Character}
A function $\vfi\in P(G)$ is a \textbf{trace} on $G$ if $\vfi$ is 
central (that is,   constant on conjugacy classes of $G$).
The set $\Tr(G)$ of traces on $G$ is a convex compact subset of 
the unit ball of $\ell^{\infty}(G)$ for the topology of pointwise
convergence. An extreme point of $\Tr(G)$ is a 
\textbf{character} of $G$. We denote by $\Char(G)$ the set 
of characters of $G.$
\end{definition}

\vskip.3cm
Several authors studied the problem of the description of $\Char(G)$
 for various discrete groups (see e.g. \cite{Thoma3}, \cite{Kirillov}, \cite{Ovcinnikov},
 \cite{Skudlarek}, \cite{Bekka},\cite{Dudko-Medynets}, \cite{Peterson},  \cite{Peterson-Thom}).

Traces on groups arise in  various  settings.
An immediate example of a trace on $G$  is the 
usual normalized trace associated to a
finite dimensional unitary representation  of $G$.
More generally, let $(\pi, \H)$ be a unitary representation of 
$G$ such that the von Neumann algebra $\pi(G)''$ of $\B(\H)$ generated by 
$\pi(G)$ admits a finite trace $\tau.$ Then $\vfi_\pi=\tau\circ \pi$
belongs to $\Tr(G)$; moreover, 
$\vfi_\pi\in \Char(G)$ if and only if $\pi(G)''$ is a factor.
Conversely, every $\vfi\in \Tr(G)$ is of the form $\vfi_\pi$
for $\pi$ as above; for all this, see  \cite[Chap. 17, \S. 3]{DixmierC*}.

Another example of a trace on $G$ is the characteristic function $\Un_N$
of a normal subgroup $N$  of $G.$
Indeed, $\Un_H$  belongs to $P(G)$ for any subgroup $H$ of $G,$ since $\Un_H$
is the matrix coefficient associated to the natural representation
of $G$ on $\ell^2(G/H)$ and the unit vector $\delta_{H}\in\ell^2(G/H)$.

In recent years, there has been a strong interest in 
the notion of invariant random subgroups, introduced in \cite{Abert}
(see e.g. \cite{AbertEtAl}, \cite{Bowen}, \cite{Glasner},\cite{Vershik1}). The space $\Sub(G)$  of subgroups  of $G$,  equip\-ped with the topology induced from $\{0, 1\}^G,$ 
is a compact space on which $G$ acts by conjugation. An \textbf{invariant
random subgroup} (or shortly IRS) of $G$  is a $G$-invariant probability measure on $\Sub(G).$  The  Dirac measures on $\Sub(G)$ are given by normal subgroups
of $G,$ which may be viewed as trivial IRS's.
Every  IRS  $\mu$ defines a trace $\vfi_\mu$ on $G$ by 
$$\vfi_\mu(g)=\mu(\{H\in \Sub(G)\mid g\in H\}) \tout g\in G;$$
indeed, $\vfi_\mu$, which is clearly $G$-invariant, belongs to $P(G)$ since
$\vfi_\mu=\int_{\Sub(G)} \Un_{H} d\mu(H)$ is a convex combination
of functions of positive type.
Examples of IRS's arise as follows: let  $G\curvearrowright (X,\nu)$ 
be a measure preserving action on a probability space $(X,\nu);$ the image $\mu$ of 
$\nu$ under the map $X\to \Sub(G), x\mapsto G_x$ is an IRS,
where $G_x$ is the stabilizer of $x$ in $G.$ In fact, as shown in \cite{Abert}, every IRS
is of this form.
An early important result in this context is the work \cite{Stuck-Zimmer}, in which 
it was  shown that, if $G$ is a lattice in a higher rank simple Lie group
and $G\curvearrowright (X,\mu)$ 
an ergodic faithful action with $\mu$ non atomic, then $G\curvearrowright (X,\mu)$ 
is essentially free (that is, $\vfi_\mu=\delta_e$). For a survey on IRS's, see \cite{Gelander}.

 \vskip.3cm
Let  $\kk$ be a field, $G$ an algebraic group defined 
and quasi-simple over $\kk,$ and $G(\kk)$  the group of $\kk$-rational points
of $G.$ Let $G(\kk)^+$ be  the (normal) subgroup of $G(\kk)$ generated by the unipotent radicals of parabolic subgroups of $G$ defined over $\kk.$ (In case $\kk$ is perfect,
$G(\kk)^+$ coincides with the subgroup generated by all unipotent elements in $G(\kk)$; see \cite[\S.6]{Borel-Tits2}.)
The celebrated Tits simplicity theorem \cite{Tits} asserts that,
when $\kk$ has at least 4 elements,  every 
subgroup of $G(\kk)$ which is normalized by $G(\kk)^+$ either 
is contained in the center $Z(G)$ of $G,$ or contains $G(\kk)^+;$  in particular, 
$$PG(\kk)^+:=G(\kk)^+/ Z(G(\kk)^+)$$
is a simple group, where  $Z(G(\kk)^+)$ is the center of $G(\kk)^+.$
We prove the following generalization of this result in the case where $\kk$ is infinite.

Given a group $H$ and a normal subgroup
$N,$ we identify $P(H/N)$  with the subset 
of $P(H)$ of all $\vfi\in P(H)$ with  $\vfi|_N=\Un_N;$
moreover, for a  function $\chi:N\to \CC,$ 
we denote by $\widetilde{\chi}$
its trivial extension to $H,$ defined by $\widetilde{\chi}=\chi$ on $N$
and $\widetilde{\chi}=0$ on $H\setminus N.$
 \begin{thmx}
  \label{Theo-CharSemisimple}
  Let $\kk$ be an infinite field and $G$ an algebraic group defined
and quasi-simple over $\kk.$ Let $\vfi\in P(G(\kk))$ be a function of positive type
on $G(\kk)$ which is conjugation-invariant under $G(\kk)^+.$
 Then $\vfi$ is a convex combination of functions $\vfi_1$ and $\vfi_2$ in $P(G(\kk))$
with $\vfi_1=1$ on $G(\kk)^+$ and $\vfi_2=0$ outside $Z(G)(\kk).$
In particular, we have
$$
\Char(G(\kk))= \Char(G(\kk)/G(\kk)^+) \cup \{\widetilde{\chi}\mid \chi \in \widehat{Z(G)(\kk)}\},$$
where $\widehat{Z(G)(\kk)}$ is the dual group of $Z(G)(\kk)$.

\end{thmx}

Our  proof of Theorem~\ref{Theo-CharSemisimple}  is not independent of Tits' simplicity theorem; its overall strategy, similar to the one for the proof
of the Howe-Moore theorem (\cite[Theorem 5.1]{Howe-Moore}),
proceeds by the analysis of the restriction of a function $\vfi$ as above to the semi-direct
product of a maximal $\kk$-split torus in $G$ with various root spaces.
In this analysis, a crucial tool is played by Proposition~\ref{CharacterSemiDirectProduct}
below, which, combined with Tits' simplicity theorem, yields a mixing result of traces on such a semi-direct product.

The following result is an immediate consequence of  
Theorem~\ref{Theo-CharSemisimple} (see Section~\ref{ExtensionReductive}).

\begin{corx}
  \label{Cor-CharSemisimple}
  Let $\kk$ and $G$ be as in  Theorem~\ref{Theo-CharSemisimple}.
  Then 
  $$
  \Char(G(\kk)^+)= \{\Un_{G(\kk)^+}\}\cup \{\widetilde{\chi}\mid \chi \in \widehat{Z(G(\kk)^+)}\}.
  $$
  In particular, we have
  $
\Char(PG(\kk)^+)= \{\Un_{PG(\kk)^+}, \delta_e\}.$
\end{corx}

 Corollary~\ref{Cor-CharSemisimple} can easily be extended to  the  case of semi-simple  groups (see Section~\ref{ExtensionReductive}).

\begin{remark}
\label{Rem-Theo-CharSemisimple}
 The result in Corollary~\ref{Cor-CharSemisimple} was obtained in the case 
 $G=SL_n$  for $n\geq 3$ in \cite{Kirillov}  
 and extended  in \cite{Ovcinnikov} to  any Chevalley group $G$  (in this case, $G$ is $\kk$-split),  not of type $A_1$ or $B_2$ (see however the footnote in {\it{op. cit.}}).
 The case $G=SL_2$ was treated in  \cite{Peterson-Thom}.
\end{remark}

As a consequence of Corollary~\ref{Cor-CharSemisimple}, we obtain the 
following rigidity result  for  factor representations of $PG(\kk)^+$ 
and for its IRS's.
\begin{corx}
  \label{Cor-CharSemisimple2}
Let $\kk$ and $G$ be as in  Theorem~\ref{Theo-CharSemisimple}.
 \begin{itemize}
 \item[(i)] Let $M$ be a von Neumann
factor which possesses a finite trace and 
$\pi: PG(\kk)^+\to U(M)$ a homomorphism in the unitary group of $M$
such that $\pi(PG(\kk)^+)''= M.$ Then either $\pi$ is the identity homomorphism
or $\pi$ extends to an isomorphism $L(G(\kk))\to M,$
where $L(PG(\kk)^+)$ is the von Neumann algebra $\la(PG(\kk)^+)''$ generated by the regular representation $\la$ 
of $PG(\kk)^+$ on $\ell^2(PG(\kk)^+).$
\item[(ii)] Assume that $\kk$ is countable. The only ergodic IRS's of $PG(\kk)^+$ are 
$\delta_{e}$ and $\delta_{PG(\kk)^+}.$
\end{itemize}
\end{corx}

\begin{remark}
\label{Rem-Cor-CharSemisimple2}
\begin{itemize}
\item[(i)] Item (i) of Corollary~\ref{Cor-CharSemisimple2} generalizes a classical result
of Segal and von Neumann \cite{Segal-VonNeumann} which deals
with the case $\kk=\RRR$, under a measurability condition on $\pi$
(see also  Item (iii) below and Proposition~\ref{Theo-Extension}).

\item[(ii)]   Assume  that $G$ is \textbf{$\kk$-isotropic}, that is,  $G(\kk)^+\neq\{e\}$
 (of course, the previous results have  content only in this case).
 Then $G(\kk)^+$ is  ``large", in the sense that 
 $G(\kk)^+$ is Zariski dense in $G$ (see \cite[3.2(20)]{Tits}).
In fact,  the so-called Whitehead group 
$W(\kk, G):= G(\kk)/G(\kk)^+$
 is known to be trivial  for various
groups  $G$ and various fields $\kk$ (for this, see \cite{Tits2}, \cite{Platonov-Rapinchuk},
\cite{Gille}). Although $W(\kk, G)$  can be non trivial (even when $G$ is simply connected), it seems that no example  is known where $W(\kk, G)$ is not abelian.

\item[(iii)] Let  $\kk$ be either a
local field (that is, a non discrete locally compact 
field) or a global field (that is, either  a number field or a function field in one variable over a finite field).  It is known  that $W(\kk, G)$ is trivial in both cases, when
$G$ is  simply connected, $\kk$-isotropic and almost $\kk$-simple (see \cite[Chap. 7]{Platonov-Rapinchuk} and \cite[Theorem 8.1]{Gille}).
It follows from this and from  \cite[Corollaire 6.5]{Borel-Tits2} that $W(\kk, G)$ is an abelian group,  when $G$ is any $\kk$-isotropic and almost $\kk$-simple group; since $G(\kk)^+$ is perfect (\cite[3.3]{Tits}), this means that 
$W(\kk, G)$ coincides with the abelianization $G(\kk)_{\rm ab}$ of $G(\kk).$
More is known about the structure of $W(\kk, G)$, when $\kk$ is a local field (see \cite[6.14, 6.15]{Borel-Tits2}); for instance, $G(\RRR)^+$ coincides with the 
identity component of the Lie group $G(\RRR).$

\item[(iv)] Observe that the Zariski-density of $G(\kk)^+$  in $G$ implies that 
the center of $G(\kk)^+$ coincides with $Z(G)\cap G(\kk)^+.$
\end{itemize}
\end{remark}
When $\kk$ is a local field, $G(\kk)$ has a natural locally compact 
topology induced by $\kk;$ however,  $\Char(G(\kk))$  in the statement of the following corollary refers to the characters of $G(\kk)$ as {\it{discrete}}\,   group.

\begin{corx}
  \label{Cor-Theo-CharSemisimple3}
  Let  $\kk$ be a local or global field and $G$ a $\kk$-isotropic and almost $\kk$-simple algebraic group. 
  
  \begin{itemize}
 \item[(i)] We have $\Char(G(\kk))= \widehat{G(\kk)_{\rm ab}} \cup \{\widetilde{\chi}\mid \chi \in \widehat{Z(G)(\kk)}\}.$
\item[(ii)]  Assume that $\kk$ is a global field and that $Z(G)=\{e\}.$ Every  measure preserving  ergodic action  of $G(\kk)$ on a probability space either factorizes through  $G(\kk)_{\rm ab}$  or is essentially free.
 \end{itemize}
\end{corx}

This paper is organized as follows. In Section~\ref{SS:PreliminaryResult}, we prove 
the  mixing result (Proposition~\ref{CharacterSemiDirectProduct}) which is crucial
for our approach.
 Some general facts about relative root spaces of
simple algebraic groups  are recalled in Section~\ref{S-FactsAlgGr}.
The proofs of Theorem~\ref{Theo-CharSemisimple} and its corollaries  are given 
in Sections~\ref{S:Proof} and ~\ref{ExtensionReductive}.

\noindent
{\bf Acknowledgement.}\ We are grateful to Jesse Peterson
for a useful suggestion concerning the proof of Corollary~\ref{Cor-Theo-CharSemisimple3}.ii.

\section{A preliminary result on traces}
\label{SS:PreliminaryResult} 
 Let $G$ be a group. Recall that a unitary representation $(\pi, \H)$ of $G$
 is a homomorphism $\pi$ from $G$ to the group of unitary operators
 on a Hilbert space $\H.$
 The von Neumann algebra $\pi(G)''$ generated by 
 $\pi(G)$ is the bicommutant of the subset $\pi(G)$ of the algebra $\B(\H)$ 
 of bounded linear operators on $\H$ and can also be described as the closure
 of the linear span of $\pi(G)$  in the strong (or the weak) operator topology. 
 
 Let $\vfi\in P(G)$. Recall (see \cite[Appendix C]{BHV})
that there exists a so-called  \textbf{GNS-triple}  $(\pi, \H, \xi)$ associated to $\vfi$ 
which consists of a unitary 
representation $\pi$ of $G$ on a Hilbert space $\H$ and 
a unit vector $\xi\in \H$ which is $G$-cyclic (that is, the linear span
of $\{\pi(g)\xi\mid g\in G\}$ is dense in $\H$) and such that
$\vfi(g)= \langle \pi(g)\xi  \mid \xi \rangle$ for every $g\in G.$
 
 The proof of the following lemma is straightforward.
 \begin{lemma}
 \label{Lem-InvariantTrace}
  Let $G^+$ be  a subgroup of $G$ and assume that  $\vfi$ is $G^+$-invariant, that is, 
 $\vfi(g xg^{-1})= \vfi(x)$ for all $x\in G$ and  $g\in G^+.$
 Let  $(\pi, \H, \xi)$ be a   GNS-triple  associated to $\vfi$
 and let $\tau_\vfi$ be the linear functional on $\pi(G)''$ defined by 
 $$
 \tau_\vfi(T)=  \langle T\xi  \mid \xi \rangle \tout T\in \pi(G)''.
 $$
 We have $\tau_\vfi(TS)= \tau_\vfi(ST)$ for all $T\in \pi(G)''$ and $S\in \pi(G^+)''.$
 \end{lemma} 

 If $\theta$ is an automorphism of an abelian locally compact group $U$,
 we denote by $\chi\mapsto \chi^\theta:=\chi\circ \theta^{-1}$ the dual action of $\theta$
 on the dual group $\widehat U$.
  
Let $(\pi, \H)$ be a unitary representation of a locally compact group $G$.
For a subset $X$ of $G,$ we denote by $\H^X$  the subspace of $X$-invariant vectors in $\H.$
Moreover, we will say that $\pi$
is mixing if the matrix coefficient $ \langle \pi(\cdot)\xi  \mid \eta \rangle$
vanishes at infinity on $G$  (that is, belongs to $C_0(G)$) for every $\xi,\eta\in \H$.

The following proposition will be a  crucial tool in our proof of Theorem~\ref{Theo-CharSemisimple}.
  \begin{proposition}
\label{CharacterSemiDirectProduct}
Let $G$ be a discrete group. Let $G^+$ be a normal subgroup of $G$ containing
  a subgroup $H$ and  an  abelian subgroup $U$ normalized by $H$
  with the following property: there exists a finite subset $F$ of $H$ such that, 
 for every  $h\in H\setminus F$ and every 
 $\chi\in \widehat U \setminus\{\Un_U\},$ we have $\chi^h\neq \chi.$
 
Let $\vfi: G\to \CC$ be a  $G^+$-invariant  normalized function of positive type on $G$
and let $(\pi, \H, \xi)$ be a  GNS-triple associated to $\vfi.$
\begin{itemize}
\item[(i)] The restriction of $\pi\vert_ H$ to  the orthogonal complement of 
 $\H^U$ is a mixing  representation of $H.$
\item[(ii)] If $H$ is infinite, then   $\H^H$ is contained in $\H^U.$
\item[(iii)]  If $\H^U=\{0\},$ then  $\varphi(hu)=0$ for every $h\in H\setminus F$ and $u\in U.$
\end{itemize}
\end{proposition}

\begin{proof} 
 Consider the projection-valued measure 
$E:{\mathcal B}(\widehat U)\to \Proj(\H)$ associated to the restriction $\pi \vert_U$
of $\pi$ to $U,$
where ${\mathcal B}(\widehat U)$ is 
the  $\sigma$-algebra of Borel subsets of the compact space $\widehat U$
and  $\Proj(\H)$ the set  of orthogonal projections on $\H$.
The commutant of $\pi(U)$ coincides with 
the commutant of $\{E(B)\mid B\in {\mathcal B}(\widehat U)\}.$
In particular, the range of $E(B)$ is $\pi(U)$-invariant 
and  $E(B) \in \pi(U)''$ for every $B\in{\mathcal B}(\widehat U).$
Moreover, the subspace $\H^U$ coincides 
with the range of the projection $E(\{\Un_U\})$ and we have
the covariance relation
$$
\pi(h) E(B) \pi(h^{-1})=E(B^{h}) \tout B \in {\mathcal B}(\widehat U), h\in H,
\leqno{(*)}
$$
where $B^h = \{\chi^h \mid \chi \in B\}$; 
for all this, see Appendix D as well as the proof of Theorem 1.4.5 in \cite{BHV}.

Set 
$$
X:=\widehat{U}\setminus \{\Un_U\}
$$
and fix $h\in H\smallsetminus F.$ Since 
$\chi^{h}\neq \chi$ for every $\chi\in X,$  there exists 
a family $(B_i)_{i \in I}$ of pairwise disjoint Borel subspaces  of $X$
with  $X=\bigcup_{i\in I} B_i$  such that
such that 
$$B_i \cap B_i^{h}=\emptyset \tout i\in I,$$
by Lemma~\ref{Lem-FreeAction} below.
By the general properties of a projection-valued measure, we have
\begin{itemize}
\item [(1)] $E(B_i)E(B_j)= 0$ for all $i \ne j,$
\item [(2)] $\sum_{i\in I} E(B_i) = E(X),$ in the strong operator topology, and
\item [(3)]  $E(B_i)E(B_i^{h}) = 0$ for all  $i\in I.$
\end{itemize}

We claim that, for every $g\in G,$ we  have
$$
\langle \pi(hu)E(X)\pi(g)\xi  \mid E(X)\pi(g)\xi \rangle=0 \tout u\in U. \leqno{(**)}
$$
Indeed, let us first prove the formula for $g=e$. Let $u\in U.$
By Lemma~\ref{Lem-InvariantTrace} and  the relation (1) above, we have
$$
\begin{aligned}
\langle \pi(hu)E(B_i)\xi \mid E(B_j)\xi \rangle &=
\tau_\vfi(E(B_j) \pi(hu) E(B_i))\\
&=\tau_\vfi (E(B_i)E(B_j) \pi(hu)) = 0,
\end{aligned}
$$
for all $i\neq j$.
By  successively relations (2), $(*)$, and (3), it follows that
$$
\begin{aligned}
\langle \pi(hu)E(X)\xi  \mid E(X)\xi \rangle&
=\langle \pi(hu)E(X)\xi  \mid E(X)\xi \rangle\\
 &= \langle \pi(hu) (\sum_{i\in I} E(B_i) \xi) \mid\sum_{i\in I} E(B_i)\xi \rangle\\
&=\sum_{i,j\in I} \langle \pi(hu) E(B_i) \xi \mid E(B_j) \xi \rangle\\
&= \sum_{i\in I} \langle \pi(hu) E(B_i) \xi \mid E(B_i) \xi \rangle\\
& = \sum_{i\in I} \langle E(B_i^{h})\pi(hu) \xi \mid E(B_i) \xi \rangle\\
&=  \sum_{i\in I} \langle E(B_i) E(B_i^{h})\pi(hu) \xi \mid \xi \rangle\\
&=0.
\end{aligned}
$$
Assume now that $g\in G$ is arbitrary and 
consider  the normalized function of positive type $\vfi^g$  on $G$
defined by $\vfi^g(x)= \vfi (g^{-1} xg)$ for $x\in G.$ 
It is clear that  $(\pi, \H, \pi(g)\xi)$ is a  GNS-triple associated to $\vfi.$
Moreover, as  $G^+$ is normal in $G,$ it is readily checked that $\vfi^g$ is 
$G^+$-invariant.  The linear functional on $\pi(G)''$   associated
to $\vfi^g$ as in  Lemma~\ref{Lem-InvariantTrace} is given by 
$\tau_{\vfi^g}(T)=  \langle T\pi(g)\xi  \mid \pi(g)\xi \rangle$
and satisfies $\tau_{\vfi^g}(ST)=\tau_{\vfi^g}(TS)$
for $T\in \pi(G)''$ and $S\in \pi(G^+)''.$
Replacing $\vfi$ by $\vfi^g$ in the previous 
argument,  we see therefore that 
$$\langle \pi(hu)E(X)\pi(g)\xi  \mid E(X)\pi(g)\xi \rangle=0,$$
for all $u\in U$; so, $(**)$ holds for every $g\in G.$

The relation $(**)$  shows that $\pi\vert_H$,  restricted to the closed linear span $V_g$ of 
$$\{\pi(h)E(X)\pi(g)\xi\mid h\in H\},$$ 
is a mixing representation of $H$, for every $g\in G .$
Observe that  $\K:=E(X)(\H)$ is the orthogonal complement of $\H^U.$
Since $\xi$ is a cyclic vector for the $G$-representation $\pi$, the linear span of 
$\{E(X)\pi(g)\xi\mid g\in G\}$ and hence the linear span of $\bigcup_{g\in G} V_g$ is dense in $\K.$ 
 
Let   $\mathcal{F}$  be the set
of all families  $(\H_i)_{i\in I}$ of pairwise orthogonal 
$\pi(H)$-invariant closed subspaces $\H_i$ of $\K$ such that 
$\pi\vert_ H$, restricted to $\oplus_{i\in I} \H_i$
is mixing. We partially order $\mathcal{F}$ by inclusion.
Observe that $\mathcal{F}$ is non empty
(since  $\{0\}\in \mathcal{F}$) and  that every chain in  $\mathcal{F}$ has  an upper bound. So, by Zorn's lemma,
 $\mathcal{F}$ has a maximal element $(\H_i)_{i\in I}.$ We claim 
 that $\K=\oplus_{i\in I} \H_i.$ Indeed, 
 assume, by contradiction, that $\K':=\oplus_{i\in I} \H_i$ is a proper subspace
 of $\K.$ Then there exists $g\in G$ such that $V_g$ is not contained in $\K'.$
 So, the orthogonal projection $P:V_g\to \K'^\perp$ is non-zero,
 where  $\K'^\perp$ is the orthogonal complement of $\K'$ in $\K.$
Since $P$ is $\pi(H)$-equivariant, there exists non-zero  
 $\pi(H)$-invariant  closed subspaces
$\K_1$ of $V_g$ and $\K_2$ of $\K'^\perp$ and a $\pi(H)$-equivariant bijective
isometry between $\K_1$ and $\K_2$  (see \cite[Proposition A.1.4]{BHV}).
Since  $\pi\vert_H$ restricted to  $V_g$ is mixing, the same holds
for the restriction of $\pi\vert_H$ to  $\K_2.$
So, $(\H_i)_{i\in I}\cup \{\K_2\}\in \mathcal{F},$ and this is a contradiction.
 This proves Item (i).

Observe that $\H^ U$ and $\K$ are $H$-invariant, since 
$H$ normalizes $U.$ 
If $H$ is infinite,  then $H$ has no non-zero invariant vector in $\K$ by (i) and hence 
$\H^H\subset \H^U.$ This proves Item (ii).

If $\H^U=\{0\},$ then $E(X)$ is the identity on $\H$ and Item (iii) follows from Equality $(**).$

\end{proof}

The following lemma was used in the proof of Proposition~\ref{CharacterSemiDirectProduct}.
\begin{lemma}
\label{Lem-FreeAction}
Let $T:X\to X$ be a homeomorphism of a 
Hausdorff topological space $X$. Assume that $T$  has no fixed point in $X.$
 There exists a  family $(B_i)_{i \in I}$ of pairwise
disjoint Borel subspaces  of $X$ with  $X=\bigcup_{i\in I} B_i$  and
such that $B_i \cap T(B_i)=\emptyset$ for every $i\in I.$
\end{lemma}
\begin{proof}
Let   $\mathcal{F}$  be the set
of all families  $\{B_i \mid {i\in I}\}$ of pairwise disjoint Borel subspaces $B_i$ of $X$ such that $B_i \cap T(B_i)=\emptyset$ for every $i\in I.$ Let $\leq$ be the partial order
on $\mathcal{F}$ given by $\{B_i \mid {i\in I}\}\leq \{B_j' \mid j\in J\}$ 
if, for every $i\in I,$ there exists $j\in J$ such that $B_i=B_j'.$

The set $\mathcal{F}$ is non empty; indeed, the family consisting of the single
subset $\{\emptyset\}$ belongs to $\mathcal{F}.$ 
  Moreover, every chain in 
 $\mathcal{F}$ has obviously an upper bound and so, by Zorn's lemma,
 $\mathcal{F}$ has a maximal element $\{B_i \mid {i\in I}\}.$

We claim that $X=\bigcup_{i\in I} B_i.$ Indeed, assume, by contradiction,
that there exists $x\in X\setminus \bigcup_{i\in I} B_i.$
Since $x\ne Tx,$ there an open neighbourhood  $U$ of $x$
 such that  $T(x)\notin U$. Consider the Borel subset $B$ of $X$ given by
$$B=\left(U\cap \left(X\setminus \bigcup_{i\in I} B_i\right)\right) \setminus {T^{-1}}(U).$$
It is clear that $T(B)\cap U=\emptyset$ and hence $T(B)\cap B=\emptyset.$
So, the family $\{B\} \cup \{B_i \mid {i\in I}\}$ belongs to $\mathcal{F}$ and 
$\{B_i \mid {i\in I}\} \leq \{B\} \cup \{B_i \mid {i\in I}\}$
Moreover, $\{B_i \mid {i\in I}\} \neq \{B\} \cup \{B_i \mid {i\in I}\},$ since  $x\in B$ and $x\notin  \bigcup_{i\in I} B_i;$ this is a contradiction.

\end{proof}

We will make repeated use of the following elementary lemma.
The commutator of  two elements $g,h$ in a group is $[g,h]= ghg^{-1}h^{-1}.$
\begin{lemma}
\label{Lem-Hilbert}
Let $G$ be a group and $\vfi: G\to \CC$  a 
function of positive type on $G$.
Let $g\in G.$ Assume that there exist a sequence
 $(g_n)_{n\geq 1}$ in $G$  such that 
 \begin{itemize}
\item  $\vfi([g^{-1},g_m]^{-1} [g^{-1},g_n])=0$ for all $m\neq n$ and 
\item $ \vfi(g_n gg_n^{-1})=\vfi(g)$ for all  $n \geq 1.$
\end{itemize}
 Then $\vfi(g)=0.$
  \end{lemma}
\begin{proof}
Let $(\pi, \H, \xi)$ be a GNS triple for $\vfi.$ Then
$$
\begin{aligned}
\langle \pi([g^{-1}, g_n])\xi,\pi([g^{-1}, g_m])\xi\rangle &=
\vfi([g^{-1}, g_m]^{-1} [g^{-1}, g_n]) =0
\end{aligned}
$$
for all $m,n$ with $m\neq n$; this shows that $(\pi( [g^{-1}, g_n])\xi)_{n\geq 1}$ 
is an orthonormal sequence and therefore weakly converges to $0$ in $\H.$
The claim follows, since, for every $n\geq1,$
$$
\begin{aligned}
\vfi(g)&=\vfi(g_ngg_n^{-1})= \vfi(g[g^{-1},g_n])=\langle \pi([g^{-1},g_n])\xi,\pi(g^{-1})\xi\rangle.
\end{aligned}
$$
\end{proof}

\section{Some facts about simple algebraic groups}
\label{S-FactsAlgGr}
 Let $\kk$ be any field, and let $G$ be a reductive algebraic group defined over $\kk.$
 Let $G(\kk)$ be  the group of $\kk$-rational points in $G.$  We need to recall a few results about the structure of $G$ and $G(\kk),$ from \cite{Borel-Book} and \cite{Borel-Tits1}.
  
  For a subgroup $H$ of $G,$  we will denote by  $H(\kk)$ the   subgroup $H\cap G(\kk)$ of $G(\kk)$.
  
 Fix a  maximal   torus $\kk$-split torus  $S$  in  $G$ 
 and a maximal torus $T$ in $G$ defined over $\kk$ and containing $S$.
 Let $\Phi=\Phi(T,G)\subset X(T)$ be the root system of $G$
 with respect to $T,$ where  $X(T)$ is the group of rational characters
 of  $T.$

 Let ${}_\kk\Phi=\Phi(S,G)\subset X(S)$ be the set of $\kk$-roots of $G$
 with respect to $S,$ that is, the  set of all non-trivial characters of  $S$ which are the restrictions to  $S$ of elements of $\Phi.$
 
 As usual, ${}_\kk \Phi^+$ (respectively, ${}_k \Phi^-$) 
denotes the set of all positive (respectively,  negative) $\kk$-roots, for a given ordering of 
$X(S).$

Let $\alpha, \beta\in  {}_\kk \Phi^+$ be distinct proportional positive  $\kk$-roots. Then
either $\beta= 2\alpha$ or $\alpha= 2\beta.$

For $\alpha \in \Phi,$ let ${U}_{\alpha}$
 be the one-parameter root subgroup of $G$ corresponding to $\alpha;$
 this is the subgroup characterized by the fact that there exists an isomorphism 
 $\theta_{\alpha}:\mathbf{G}_a\to {U}_\alpha$
 such that 
 $$t \theta_{\alpha}(x) t^{-1}= \theta_{\alpha}(t^\alpha x)
 \tout t\in T, x\in \mathbf{G}_a,$$
where $\mathbf{G}_a$ is the additive group of dimension 1
and where we used the exponential notation to denote the action of $X(T)$ on $T.$

Let $\alpha\in {}_\kk \Phi .$   We denote by $U_{(\alpha)}$ be 
 the group generated by all ${U}_{\beta}$ for $\beta\in \Phi$
 such that the restriction of $\beta$ to  $S$ is a positive integer multiple of $\alpha.$
 
 Set $\alpha'= \alpha$ if  $2\alpha$ is not a $\kk$-root and $\alpha'= 2\alpha$ if  $2\alpha$ is a $\kk$-root. Then $U_{(\alpha')}$  and $U_{(\alpha)}/ U_{(\alpha')}$ have a vector space structure over $\kk$ which coincides with their group structure, such that the action of $t\in S$ on  $U_{(\alpha')}$ and $U_{(\alpha)}/ U_{(\alpha')}$ is given by   scalar  multiplication by $t^{\alpha'}$ and $t^\alpha,$
 respectively. When $2\alpha$ is a $\kk$-root, $U_{(\alpha)}$ is not commutative and  its center is $U_{(\alpha')};$  moreover, $U_{(\alpha)}$
 is $\kk$-isomorphic, as variety, to an affine space.

Let $U^+$ (respectively  $U^-$) be the 
the subgroup generated by all the $U_{(\alpha)}$ for $\alpha\in {}_\kk\Phi^+$
(respectively, $\alpha\in {}_\kk\Phi^-$).
Let $Z_G(S)$ be the centralizer of $S$ in $G$
and $N_G(S)$ its normalizer.
The subgroups  
$$P^+=Z_G(S)\ltimes U^+ \text{ and } P^-=Z_G(S)\ltimes U^-$$
 are minimal $\kk$-parabolic subgroups of $G$ with unipotent radical $U^+$ and 
$U^-$, respectively;
moreover, we have $P^+ \cap P^{-}=  Z_G(S)$ .
All the subgroups introduced above are defined over $\kk.$

The  Weyl group $_{\kk}W= N_G(S)/Z_G(S)$ is finite and every element 
$w\in _{\kk}W$ has a representative in  $N_G(S)(\kk),$  denoted by the same letter $w.$
So, 
$$_{\kk}W= N_G(S)(\kk)/Z_G(S)(\kk);$$
observe that
 $_{\kk}W$  acts on $S$, on $X(S)$, and linearly on the real vector space $X(S)\otimes \RRR.$ We equip $X(S)\otimes \RRR$ with an invariant inner product.
 
The set  ${}_\kk\Phi$ of $\kk$-roots  is a (possibly non reduced) irreducible root system in $X(S)\otimes \RRR,$ whose Weyl group is isomorphic to $_{\kk}W$.

We recall the following fundamental result (see \cite[Theorem 21.15]{Borel-Book}).
\begin{theorem}
\label{Theo-BruhatDec}
\textbf{(Bruhat decomposition)} 
We have 
$$
G(\kk)=\coprod_{w\in _{\kk}W} P^+(\kk) w P^+(\kk)=\coprod_{w\in _{\kk}W} P^-(\kk) w P^-(\kk).
$$
\end{theorem}
 The following proposition is a special case of \cite[Proposition 3.11]{Borel-Tits1}.
 
 The set of \emph{non-divisible  roots}
 is the subset   ${}_{\kk} \Phi_{\rm nd}$ of  ${}_{\kk} \Phi$ of roots $\alpha$ such that $\alpha/2$ is not a root.
\begin{proposition}
 \label{Prop-UnipotentDirectGen}
 Let $H$ be a connected algebraic subgroup of $U^+$
 normalized by the torus $S$. Then 
 the multiplication map 
 $$H\cap U_{(\alpha_1)}\times \dots \times H\cap U_{(\alpha_n)} \to H $$is an isomorphism of algebraic varieties,
where $\alpha_1, \dots, \alpha_n$ is any ordering of 
 the set $ {}_{\kk} \Phi^+_{\rm nd}$.
  In particular, $U^+$ is $\kk$-isomorphic, as variety, to an affine space.
\end{proposition}

We will need to determine the structure of the center $Z(U^+)$ of the unipotent radical $U^+$ of $P^+.$

 The set of \emph{non-multipliable roots}
 is the subset   ${}_{\kk} \Phi_{\rm nm}$ of  ${}_{\kk} \Phi$ of roots $\alpha$ such that $2\alpha$ is not a root; so
$${}_{\kk} \Phi_{\rm nm}= \{ \alpha'\mid \alpha\in {}_{\kk} \Phi\},$$ 
where $\alpha'$ is defined as above for a $\kk$-root $\alpha.$
It is clear that  ${}_{\kk} \Phi_{\rm nm}$ is an irreducible and  reduced  root system in $X(S)\otimes \RRR.$

Let $R$ be an irreducible and  reduced  root system in an euclidean space. Recall that $R$
has two lengths of roots when $R$ is of type  $B_n, C_n$ for $n\geq 2$ as well
 as $F_4$ and  $G_2$ and that $R$  has one length of roots in all other cases
 (see \cite[Chap. VI]{Bourbaki}). 
 \begin{proposition}
 \label{Prop-CenterUnipotent}
 Let $\alpha_l'$ be the highest positive root of the root system ${}_{\kk} \Phi_{\rm nm};$
 when ${}_{\kk} \Phi_{\rm nm}$ has two lengths of roots,
 let $\alpha'_s$  be its highest short positive root.
 \begin{itemize}
\item We have
 $Z(U^+)= U_{(\alpha_l')}U_{(\alpha_s')}$,
 when the characteristic of $\kk$ is $2$ and ${}_{\kk} \Phi$
 of type  $B_n, C_n$ or $ F_4$ as  well as when the characteristic of $\kk$ is $3$ and $_{\kk} \Phi$ of type  $G_2.$ 
 \item We have
 $Z(U^+)= U_{(\alpha_l')}$
 in all other cases.
 \end{itemize}
 In particular, the algebraic subgroup $Z(U^+)$  is  connected and  defined over $\kk.$
 \end{proposition}
\begin{proof}
The result  is proved in \cite[Proposition 8.3]{Pink-Larsen} in
 the case where $S=T$ (that is, when $G$ is $\kk$-split).
The proof in this case is based on an analysis of subsystems of rank 2 of
the root system $\Phi=\Phi(T,G),$ in combination with
the fact that $Z(U^+)$  is the direct product, as algebraic variety,
of  the root spaces it contains.

Denote by $H$ the subgroup generated by $\{U_{(\alpha)}\mid \alpha\in{}_{\kk} \Phi_{\rm nm}^+\}$. Observe that the center $Z(H)$ of $H$, being a characteristic subgroup of $H,$ is normalized by  $S$. Using Proposition~\ref{Prop-UnipotentDirectGen}, the proof of \cite[Proposition 8.3]{Pink-Larsen} extends  without any change  to the root system ${}_{\kk} \Phi_{\rm nm}$
and shows that the connected component  $Z(H)_0$
of $Z(H)$  is 
either  $U_{(\alpha_l')}$ or  $U_{(\alpha_l')} U_{(\alpha_s')},$
depending on the type of ${}_{\kk} \Phi_{\rm nm}$ and the characteristic of $\kk,$
as described above. 
Since  $U_{(\alpha')}$ is the center of $U_{(\alpha)}$ for every $\alpha\in  {}_{\kk} \Phi$,
we have $Z(H)=Z(U^+).$

It remains to show that $Z(U^+)$ is connected.
Let $U^*$ be the product of the  $U_{(\alpha)}$, for $\alpha \in {}_{\kk} \Phi^+_{\rm nd}$ 
with  $\alpha'\neq \alpha_l'$ or $\alpha' \notin\{\alpha_l', \alpha_s'\}.$
By Proposition~\ref{Prop-UnipotentDirectGen} again and the result just proved,
 the multiplication map  
 $$Z(U^+)_0 \times U^* \to U^+$$
 is an isomorphism of varieties. Since  $Z(U^+)_0 $ has finite index in $ Z(U^+) $, the subgroup
 $U^*\cap Z(U^+) $ is finite. 
 As $S$ is connected, it follows that $S$ acts trivially on  $U^*\cap Z(U^+)$.
 This implies  that  $U^*\cap Z(U^+)=\{e\}$ and therefore  $Z(U^+)=Z(U^+)_0.$
\end{proof}
\begin{remark}
\label{Rem-Prop-CenterUnipotent}
(i) With the notation as in Proposition~\ref{Prop-CenterUnipotent},
assume that the root system ${}_{\kk} \Phi$ is not reduced. Then 
$ {}_{\kk} \Phi$ is of type $BC_n$ and so has three lengths of roots. As can be checked,
 the  root system  ${}_{\kk} \Phi_{\rm nm}$  is of type  $C_n$ and its short roots are the roots of   $ {}_{\kk} \Phi$ of intermediate length.
 
 \vskip.3cm
 \noindent 
 (ii) The same description  as in  Proposition~\ref{Prop-CenterUnipotent} holds 
 for the  center $Z(U^-)$ of $U^-,$ with $U_{(-\alpha_l')}$
 and $U_{(-\alpha_s')}$ in place of $U_{(\alpha_l')}$
 and $U_{(\alpha_s')}$.
\end{remark}

We will need to consider reductive subgroups of $G$ attached to a $\kk$-root.
\begin{proposition}
\label{Pro-RedRankOne}
Let $\alpha\in {}_k\Phi$ 
and  $G_{(\alpha)}$ the subgroup of $G$ generated by 
$U_{(\alpha)},Z_G(S)$ and $U _{(-\alpha)}.$
\begin{itemize} 
\item[(i)] $G_{(\alpha)}$ is a  reductive group defined over $\kk$
and the set  $\Phi(S, G_{(\alpha)})$ of $\kk$-roots of $S$ on  $G_{(\alpha)}$  
is  $\ZZ\alpha \cap {}_k\Phi.$
\item[(ii)] Let $\sigma_{\alpha}\in {}_{\kk}W$ be the reflection with respect to the hyperplane orthogonal to $\alpha.$
Then $\sigma_{\alpha}$ has a  representative in $ N_G(S)\cap G_{(\alpha)}(\kk)$ and we have 
$$
G_{(\alpha)}(\kk)= P_\alpha^+(\kk)
\amalg U_{(\alpha)}(\kk)\sigma_{\alpha}\ P_\alpha^+(\kk)=  P_\alpha^-(\kk)
\amalg U_{(-\alpha)}(\kk)\sigma_{\alpha}\ P_\alpha^-(\kk)
 $$
 where $P_\alpha^{\pm}=Z_G(S)U_{(\pm \alpha)}.$
\end{itemize}
\end{proposition}
\begin{proof}
Item (i)  follows from Theorem 3.13 in \cite{Borel-Tits1}, applied to the closed subset 
$ \{\beta \in \Phi(T,G)\mid \beta|_S\in  \ZZ\alpha\}$ 
of $\Phi(T, G)$.

The Weyl group  of $\Phi(S, G_{(\alpha)})$ is $\{e, \sigma_\alpha\}.$
Since $N_{G_{(\alpha)}}(S)$ coincides with $N_G(S)\cap G_{(\alpha)},$
Item (ii) follows from  the result about the Bruhat decomposition (Theorem~\ref{Theo-BruhatDec}), applied to  the reductive group $G_{(\alpha)}.$
\end{proof}

Let  $G(\kk)^+$ be the subgroup of $G(\kk)$ generated by 
the unipotent radicals of  parabolic subgroups of $G$ defined over $\kk.$
Then $G(\kk)^+$ is the subgroup generated by $U^+(\kk)$ and $U^-(\kk)$
(\cite[Proposition 6.2]{Borel-Tits2}).

Let $X_*(S)$ denote the set of multiplicative one-parameter subgroups in $S,$
that is, the set of morphisms $\mathbf{GL}_1\to S$ of algebraic groups.
Elements  in $X(S)$ and $X_*(S)$ are defined over $\kk$ and both $X(S)$ and $X_*(S)$ are free abelian groups of rank equal to $\dim S$. 
The Weyl group ${}_k W$ acts in the obvious way on $X_*(S).$

For $\chi\in X(S), \la\in X_*(S),$ we have
 $\chi\circ \la \in X(\mathbf{GL}_1)$ and so $\chi\circ \la$ is of the form
 $x\mapsto x^r$ for some $r\in \ZZ.$ The map 
$X(S)\times X_*(S)\to \ZZ$ defined by 
$$
\langle \chi, \la \rangle= r \, \text{ if }\,  \chi\circ \la (x)=x^r.
$$
is a duality pairing of $\ZZ$-modules (see \cite[Chap.III, 8.6]{Borel-Book}).

Let $\la\in X_*(S), \alpha\in {}_k\Phi,$ and 
$r=\langle \alpha, \la \rangle.$
Set $r'=r$ if $\alpha$ is non-multipliable and $r'=2r$ otherwise.
For $x\in \kk^*,$ the action of $\la(x)$  on $U_{(\alpha')}$  and $U_{(\alpha)}/ U_{(\alpha')}$, endowed with the vector space structure mentioned above, is given  by
$\la(x)\cdot v= x^{r'} v$ for $v\in U_{(\alpha')}$ 
and  by  $\la(x)\cdot v= x^{r} v$ for  $v\in U_{(\alpha)}/ U_{(\alpha')}.$ 

\section{Proof of Theorem~\ref{Theo-CharSemisimple}}
\label{S:Proof}
Let $G(\kk)$ be the group of $\kk$-points of an algebraic group $G$  defined and quasi-simple over a field $\kk$. We assume that $\kk$ is  infinite and, to avoid trivialities,
that $G(\kk)^+\neq \{e\},$ that is, the maximal $\kk$-split torus $S$  has dimension $>0.$

\subsection{Proof of the first statement}
Let $\vfi:G(\kk)\to \CC$ be a normalized function of positive type. 
Assume that $\vfi$ is $G(\kk)^+$-invariant  and let  $(\pi, \H, \xi)$ be a GNS-triple associated to $\vfi.$

Since $G(\kk)^+$ is normal in $G(\kk),$ 
the closed subspaces 
$$\K_1:=\H^{G(\kk)^+} \quad \text{ and } \quad \K_2:= \K_1^\perp$$
are $\pi(G)$-invariant.
For $i=1,2,$ let $\pi_i$ be the corresponding subrepresentations of $\pi$ defined
on $\K_i$. (Observe that $\K_i$ may be $\{0\}$.)
Writing $\xi= \xi_1 + \xi_2$ with 
$\xi_i \in \K_i,$
we have therefore 
$$\vfi= \vfi_1+\vfi_2,$$ where
$\vfi_i$ is the (not necessarily normalized) function of positive type on 
$G$ given by $\vfi_i(g)= \langle \pi(g)\xi_i\mid \xi_i \rangle.$
It is clear that $\vfi_1$ and hence $\vfi_2$ are $G(\kk)^+$-invariant.
Obviously, we have $\vfi_1|_{G(\kk)^+}=\Un_{G(\kk)^+}$.
In case $\K_i\neq \{0\},$ we have $\xi_i\neq 0$ and the triple $(\pi_i, \K_i, t_i \vfi_i)$ is a GNS-triple for the normalized function of positive type $t_i \vfi_i$, where $t_i=\Vert \xi_i\Vert^{-2}.$

As a result of the previous discussion, upon replacing $\vfi$ by $\vfi_2,$ we may and will assume in the sequel 
that  $\H^{G(\kk)^+}=\{0\}$. 
We have then to show that  $\vfi=0$ outside $Z(G)(\kk)$.

\vskip.3cm

For $\la\in X_*(S)$, we set 
$$H_\la:=\la(\mathbf{GL}_1)\cap G(\kk)^+.$$
Assume that $\la$ is non trivial. Then $\la(\mathbf{GL}_1)$  is a one-dimensional subtorus of $S$; since $\kk$ is infinite, the subgroup $H_\la$
 of $G(\kk)^+$   is  Zariski-dense in $\la(\mathbf{GL}_1)$
(see  \cite[Corollaire 6.8]{Borel-Tits2}); in particular, $H_\la$   is infinite. 
 Let  $\alpha\in {}_k\Phi$  and $r:=\langle \alpha, \la\rangle \in \ZZ.$
 Observe that if $r\ne 0,$ then 
$$\la(\kk^*)\cap \ker \alpha= \{\la(x)\mid x\in \kk^* \text{ with } \ x^r=e\}$$
and hence $H_\la \cap \ker \alpha$  is a finite subset of $H_\la.$

 \vskip.2cm
$\bullet$ {\it First  step.}  Let  $\alpha\in {}_k\Phi_{nm}$ be a non-multipliable root.
Then there exists   a one-parameter subgroup $\la\in X_*(S)$ such that 
$$
\H^{U_{(\alpha)}(\kk)}= H^{H_\la} \qquad\text{and} \qquad
  \H^{U_{(-\alpha)}(\kk)}= H^{H_\la}.
  $$
Indeed, by \cite[3.1(13)]{Tits} (see also the beginning of the proof of Th\'eor\`eme 7.2  in \cite{Borel-Tits1}),   there exists a  homomorphism
$\rho_\alpha: SL_2\to  G$ defined over $k,$  with kernel contained in the center of 
$SL_2,$ such that 
$$
\rho_\alpha
\begin{pmatrix}
1 & t\\ 0 & 1
\end{pmatrix} \in U_{(\alpha)}, 
\hskip.2cm 
\rho_\alpha\begin{pmatrix}
1 & 0 \\ s& 1
\end{pmatrix} \in U_{(-\alpha)}, 
\hskip.2cm 
\text{and} \hskip.2cm 
\rho_\alpha\begin{pmatrix}
x & 0 \\ 0& x^{-1}
\end{pmatrix} \in S.
$$
Let  $L:= \rho_\alpha(SL_2(\kk))$ and observe that 
 $\vfi\circ (\rho_\alpha|_{SL_2(\kk)})$ is a central function of positive type
 on $SL_2(\kk).$
It follows from the classification of the characters of $SL_2(\kk)$ in \cite[Theorem 2.4]{Peterson-Thom} that 
$$\vfi|_{L}=  t \Un_{L}+(1-t) \vfi'$$
for some $t\in [0,1],$ 
where $\vfi'$ is a function of positive type on $L$ with 
$\vfi'=0$ outside the (finite) center of $L.$
Since a similar statement holds for the conjugate function $\vfi^g$ for every $g\in G(\kk),$
this implies (compare with the proof of Proposition~\ref{CharacterSemiDirectProduct}) 
that the restriction of $\pi|_{L}$ to the orthogonal complement of 
$\H^{L}$ is a mixing representation of $L.$ 
Let $\la \in X_*(S)$ be defined by $\la(x)= 
\rho_\alpha\begin{pmatrix}
x & 0 \\ 0& x^{-1}\end{pmatrix}$
and let 
$$V:=  \rho_\alpha \left(\left\{\begin{pmatrix}
1 & t\\ 0 & 1
\end{pmatrix}\mid t\in \kk\right\}\right).
$$
Since  the subgroups $H_\la$ and  $V$ of $L$ are infinite and 
since $\pi|_{L}$ is mixing on $(\H^{L})^\perp,$ we have 
$$\H^{L}= \H^{H_\la}=  \H^{V}.$$

Observe that  $\langle \alpha, \la\rangle=2\neq 0$ and so
 $F:=H_\la \cap \ker \alpha$ is finite.
Recall that the subgroup $U_{(\alpha)}(\kk)$ has the structure of a $\kk$-vector space.
 For  $\chi\in  \widehat{U_{(\alpha)}(\kk)}$ and $x\in \kk^*,$ we have
$$\chi^{\la(x)}(u)=\chi(\la(x)^{-1}\cdot u))=  \chi(x^{-2} u) \tout  u\in U_{(\alpha)}(\kk).$$
Thus,  for every $t\in H_\la\setminus F$ and every
 $\chi\in \widehat{U_{(\alpha)}(\kk)}$ with $\chi\neq\Un_{U_{(\alpha)}(\kk)},$ we have $\chi^t\neq \chi$. Proposition~\ref{CharacterSemiDirectProduct},
applied to the subgroup $H_\la U_{(\alpha)}(\kk)$ of $ G(\kk)^+$,
shows then that  $\H^{H_\la}$ is contained in
$\H^{U_{(\alpha)}(\kk)}$. 

It is obvious that $\H^{U_{(\alpha)}(\kk)}$ is contained in $\H^{V}$
and it was shown before that $\H^{V}= \H^{H_\la}$. As a result, we obtain that
$\H^{U_{(\alpha)}(\kk)}= \H^{H_\la}$.

The same proof  shows also that  $\H^{U_{(-\alpha)}(\kk)}= \H^{H_\la}$.

\vskip.2cm
$\bullet$ {\it Second  step.} 
Let  $\alpha\in {}_k\Phi$ be such that $U_{(\alpha)}$ is either contained in the center
$Z(U^+)$ of $U^+$ or in the center $Z(U^-)$ of $U^-$.
We claim that $\H^{U_{(\alpha)}(\kk)}$ is $G(\kk)$-invariant.

Indeed, assume that  $U_{(\alpha)}\subset Z(U^+),$ that is, $\alpha \in  {}_k\Phi^+.$
Observe that $\alpha$ is non-multipliable (see Proposition~\ref{Prop-CenterUnipotent}).
By the first step, there exists a  one-parameter subgroup
$\la\in X_*(S)$ such that $\H^{U_{(\alpha)}(\kk)}= \H^{U_{(-\alpha)}(\kk)}=H^{H_\la}.$

Observe that $\H^{H_\la}$ is  obviously invariant under $Z_G(S)(\kk)$.
Since  $G(\kk)$ is generated by 
$U^+(\kk), U^-({\kk})$ and $Z_G(S)(\kk),$ it suffices therefore to show
that $\H^{U_{(\alpha)}(\kk)}$ is invariant under  $U_{(\alpha')}(\kk)$ for 
every $\kk$-root $\alpha'.$ 

Let $\alpha' \in  {}_k\Phi^+.$
Since
$U_{(\pm\alpha)}\subset Z(U^\pm),$ it is obvious that $\H^{U_{( \alpha)}(\kk)}$ is invariant under  $U_{(\alpha')}(\kk)$ and that $\H^{U_{(-\alpha)}(\kk)}$ is invariant under  $U_{(-\alpha')}(\kk)$. Since 
$H^{U_{(\alpha)}(\kk)}=H^{U_{(-\alpha)}(\kk)},$  the claim follows
in the case where $U_{(\alpha)}$ is contained in $ Z(U^+).$

 The case where $U_{(\alpha)}$ is contained in $ Z(U^-)$  is entirely 
 similar.

\vskip.2cm
$\bullet$ {\it Third step.} Let  $\alpha\in {}_k\Phi$ be such that $U_{(\alpha)}$ is either contained in $Z(U^+)$ or in  $Z(U^-)$.  We claim that  $\H^{U_{(\alpha)}(\kk)}=\{0\}.$

Indeed, assume, by contradiction, that $\H^{U_{(\alpha)}(\kk)}\neq \{0\}.$ By the second step, 
$\H^{U_{(\alpha)}(\kk)}$ is $G(\kk)$-invariant and so defines a  subrepresentation 
$\rho$ of $\pi.$ The kernel $L$ of $\rho$ is a normal subgroup of $G(\kk)$
containing the non-central subgroup $U_{(\alpha)}(\kk).$  It follows from 
Tits' simplicity theorem (\cite{Tits}) that $L$ contains $G(\kk)^+.$
This is a contradiction, since $\H^{G(\kk)^+}=\{0\},$ by our assumption.

\vskip.2cm

In the sequel, we will prove that $\vfi(g)=0$ for elements 
$g$ from various subsets of $G(\kk);$ for this, we will repeatedly use  Lemma~\ref{Lem-Hilbert} by finding sequences $(g_n)_n$ in $G(\kk)^+$ such
that $\vfi([g^{-1},g_m]^{-1}[g^{-1},g_n])=0$ for all $m\neq n.$ 

\vskip.2cm
$\bullet$ {\it Fourth step.} Let  $\alpha\in {}_k\Phi$ be such that $U_{(\alpha)}$ is either contained in $Z(U^+)$ or in  $Z(U^-)$. We claim that 
$\vfi(g)=0$ for every $g\in U_{(\alpha)}(\kk)$ with $g\neq e.$

Indeed, consider   the subgroup $G_{(\alpha)}$ generated by 
$U _{(\alpha)} ,Z_G(S)$ and  $U_{(-\alpha)}$.
We have
$$G_{(\alpha)}(\kk)= P_\alpha^-(\kk) \amalg U_{(-\alpha)}(\kk)\sigma_{\alpha} P_\alpha^-(\kk),
 $$
 where  $P_\alpha^{-}=Z_G(S)U_{(-\alpha)}$ and 
  $\sigma_{\alpha}$ is the reflection which sends $\alpha$ to $-\alpha$
 (see Proposition~\ref{Pro-RedRankOne}).
 
 Assume that  $U_{(\alpha)}\subset Z(U^+)$ and
 let $g\in U_{(\alpha)}(\kk)\setminus\{e\}.$
 Since  $P^+_{\alpha} \cap P_\alpha^{-}= Z_G(S),$
 we have  $g\notin P_\alpha^{-}$ and hence  
 $$g\in U_{(-\alpha)}(\kk)\sigma_{\alpha} P_\alpha^-(\kk).$$
 So, the conjugate of $g$ by some element 
 from $U_{(-\alpha)}(\kk)\subset G(\kk)^+$ belongs to $\sigma_{\alpha} P_\alpha^-(\kk)$.
 We may therefore assume that $g$ is of the form
 $$
 g= \sigma_{\alpha} \ga u \,\,  \text{ for some } \, \ga\in Z_G(S)(\kk)\, \text{ and  } \, u\in U_{(-\alpha)}(\kk).
 $$ 
 Choose $\la\in X_*(S)$  such that $r:=\langle \alpha, \la\rangle \neq 0$.
The corresponding subgroup $H_{\la}$ of $G(\kk)^+$ is infinite and we can therefore find
a sequence $(x_n)_{n\geq 1}$ in $\kk^*$
such that $\la(x_n) \in G(\kk)^+$  and such that
$$x_m^{2r}\neq x_n^{2r} \tout m\neq n.$$
Consider the multiplicative one-parameter subgroup 
$$\la':= \la^{\sigma_{\alpha}} \la^{-1} \in X_*(S)$$
Observe that  $\langle \sigma_{\alpha}(\alpha), \la\rangle =-r$ and 
hence
$$
 \langle \alpha, \la'\rangle= \langle  \sigma_{\alpha}(\alpha), \la\rangle
-\langle  \alpha, \la\rangle=-2r.
$$
Therefore, we have 
$$
\leqno{(*)}\qquad \la'(x_m)^{-1} \la'(x_n)\notin \ker(-\alpha) \tout m\neq n.$$
Moreover, 
$$\la'(x_n)=\la^{\sigma_{\alpha}}(x_n) \la^{-1}(x_n) \in H_{\la'} \tout n\geq 1,$$
since $G(\kk)^+$ is a normal  subgroup of $G(\kk).$ 
By the third step,  there is no non-zero  $U_{(-\alpha)}(\kk)$-invariant vector in 
$\H$. Therefore, Proposition~\ref{CharacterSemiDirectProduct} applied to 
$H_{\la'} U_{(-\alpha)}(\kk)$  shows that  
$$
\leqno{(**)} \qquad \vfi(hu)=0 \tout h\in H_{\la'}\setminus \ker(-\alpha)\, \text{ and }\,  u\in U_{(-\alpha)}(\kk).
$$
For every $n\geq 1,$ we have
$$
 [\sigma_{\alpha}^{-1}, \la(x_n)]= 
  \la(x_n)^{\sigma_{\alpha}}\la(x_n)^{-1}= \la'(x_n).
  $$
Set $g_n:=[g^{-1},  \la(x_n)];$ then 
  $$g_n=u^{-1}\ga^{-1}[\sigma_{\alpha}^{-1}, \la(x_n)]  \ga \left(\la(x_n)\cdot u\right)= u^{-1}\la'(x_n)  x_n^{-r}u.$$ 
In view of $(*)$ and $(**)$, we have therefore
$$
\begin{aligned}
\vfi(g_m^{-1} g_n)= \vfi\left(\la'(x_m)^{-1}\la'(x_n) ((x_n^{-r}-x_m^{-r}) u)\right)= 0
\end{aligned}
$$
for all $n\neq m$ and the claim follows.

The case where  $U_{(\alpha)}\subset Z(U^-)$ is treated similarly.

\vskip.2cm
$\bullet$ {\it Fifth step.} We claim that 
$\vfi(g)=0$ for every $g\in Z(U^\pm)(\kk)$ with $g\neq e.$

We will only  treat the case of $Z(U^+),$ 
the case of $Z(U^-)$ being entirely similar.

By Proposition~\ref{Prop-CenterUnipotent},
we have either $Z(U^+)(\kk)= U^+_{(\alpha)}(\kk)$
or  $Z(U^+)(\kk)= U^+_{(\alpha)}(\kk)U^+_{(\beta)}(\kk)$
for distinct roots  $\alpha, \beta\in {}_k\Phi^+_{\rm nm}$. 
In the first case, the claim follows from the fourth step.

So, we may assume that $Z(U^+)(\kk)= U_{(\alpha)}(\kk)U_{(\beta)}(\kk)$
for roots $\alpha\neq \beta$  as above.
Let $g\in Z(U^+)(\kk)$ with $g\neq e.$ 
Then $g=uv$ for $u\in U_{(\alpha)}(\kk)$ and $v\in U_{(\beta)}(\kk)$

If $u=e,$ then $g\in  U_{(\beta)}(\kk)$ and therefore $\vfi(g)=0,$
by the fourth step. We can therefore  assume that $u\neq e.$

Since $\alpha$ and $\beta$ are non proportional, we can find
a one-parameter subgroup $\la\in X_*(S)$ such that 
$$
\langle \beta, \la\rangle= 0\,\, \text{ and } \,\,\langle \alpha, \la\rangle=r\neq 0.
$$
Choose  a sequence $(x_n)_{n\geq 1}$ in $\kk^*$
such that $\la(x_n) \in G(\kk)^+$ and such that 
$$x_m^{r}\neq  x_n^{r} \tout m\neq n.$$
We have 
$$
 \la(x_n)\cdot v= v \, \text{ and }\,   \la(x_n)\cdot u= x_n^r u.
$$
and, with  $g_n:=[g^{-1}, \la(x_n)],$ 
$$g_n= v^{-1}u^{-1}(x_n^r u)v \tout n\geq 1.$$
Therefore, we have, by the fourth step,
$$
\vfi(g_m^{-1} g_n)= \vfi(v^{-1} (x_m^r u)^{-1} (x_n^r u) v)=\vfi((x_n^{r}- x_m^{r})u)= 0
$$
for all $m\neq n$  and the claim follows.

\vskip.2cm
For subgroups $A$ and $B$ of a group, $[A,B]$ denotes the subgroup generated by the commutators $[a,b]$ for $a\in A, b\in B.$

Recall that   the descending central series $(\mathcal{C}^i(U^+))_{0\leq i\leq n}$ of 
the nilpotent group $U^+$ is inductively defined by 
$$
\mathcal{C}^0(U^+)= U^+ \, \text{and } \,\, \mathcal{C}^{i+1}(U^+)= [U^+, \mathcal{C}^i(U^+)] \text{ for } 0\leq i\leq n-1,
$$
where $n$ is  the smallest integer $i\geq 1$ with $\mathcal{C}^{i}(U^+)=\{e\}.$
Every $\mathcal{C}^i(U^+)$ is an algebraic normal subgroup of $U^+$ defined over $\kk.$

\vskip.2cm
$\bullet$ {\it Sixth step.}  We claim that 
$\vfi(g)=0$ for every $g\in U^\pm(\kk)$ with $g\neq e.$ 

We will only treat the case of $U^+.$
Let $i\in \{1, \dots, n-1\}$. Assume that 
$$\vfi(u)=0 \tout u\in \mathcal{C}^{i+1}(U^+)(\kk)\setminus \{e\}.$$
Let $g\in \mathcal{C}^{i}(U^+)(\kk)$.
We are going to show that  $\vfi(g)=0$.
Once proved,  the claim will follow since 
$\mathcal{C}^{0}(U^+)=U^+$ and  $\mathcal{C}^{n}(U^+)=\{e\}.$

In view of the fifth step, we can assume that  $g\notin Z(U^+).$ 
The map
$$
f:U^+\to U^+, \quad u\mapsto [g^{-1}, u]
$$
is morphism of  algebraic varieties and is defined over $\kk$.  Since $g\in  \mathcal{C}^{i}(U^+)$ and $g\notin Z(U^+),$
 the image $f(U^+)$  of $f$ is contained in $\mathcal{C}^{i+1}(U^+)$ and is distinct from $\{e\}.$
 The Zariski-closure of $f(U^+)$ is an irreducible $\kk$-subvariety of $U^+$
 (see proof of \cite[2.2 Proposition]{Borel-Book}). In particular, the set 
 $f(U^+)$ is infinite.
 
Since ${U}^+(\kk)$ is Zariski dense in ${U}^+$
(see Proposition~\ref{Prop-UnipotentDirectGen}
or Theorem 21.20 (i) in \cite{Borel-Book}), 
 we can find a sequence $(u_n)_{n\geq 1}$ in ${U}^+(\kk)$ such that 
$f(u_m)\neq f(u_n)$, that is, such that
$$
[g^{-1}, u_m] \neq [g^{-1}, u_n] \tout n\neq m.
$$
Then 
$$[g^{-1}, u_m]^{-1}[g^{-1}, u_n] \in \mathcal{C}^{i+1}(U^+)(\kk)\setminus \{e\};
$$
hence, by our assumption, we have
$$
\vfi([g^{-1}, u_m]^{-1}[g^{-1}, u_n])=0 \tout m\neq n
$$
and the claim follows.
 
\vskip.2cm
$\bullet$ {\it Seventh step.} Let $t\in S(\kk)$,
with $t\notin Z(G).$
We claim that  $\vfi(t)=0$.

Since $t\notin Z(G),$ there exists $\alpha\in {}_k\Phi$
 such that  $t^\alpha\neq 1;$ indeed,  otherwise 
 $t$ would centralize $U^+, U^-$ and $Z_G(S)$ and would
 therefore belong to $Z(G)(\kk).$
 
Since $\kk$ and hence $U_{(\alpha)}(\kk)$ is infinite,
we can find   a sequence $(u_n)_{n\geq 1}$ in $U_{(\alpha)}(\kk)$ with 
$u_n\neq u_m$ for $n\neq m.$
We are going to show  that 
$$
 [t^{-1}, u_m]\neq [t^{-1}, u_n] \tout m\neq n.
 $$
Assume first that $\alpha$ is non-multipliable.
Then 
$$
[t^{-1}, u_m]= (t^{-\alpha}-1)u_m\neq (t^{-\alpha}-1)u_n= [t^{-1}, u_n] \tout m\neq n.
$$
Assume next that $2\alpha$ is a root. We see in a similar way
that the images of $[t^{-1}, u_m]$ and $[t^{-1}, u_n]$ in  $U_{(\alpha)}/U_{(2\alpha)}$ 
are distinct. 

Therefore, by the fourth step, we have
 $$\vfi ([t^{-1}, u_m]^{-1}[t^{-1}, u_n])=0 \tout m\neq n
 $$ 
and the claim follows.

\vskip.2cm
$\bullet$ {\it Eighth step.} Let $g\in S(\kk)U^\pm(\kk)$ with $g\notin Z(G).$
We claim that  $\vfi(g)=0$.

We have $g=t u$ for some $t\in S(\kk)$ and $u\in U^\pm(\kk).$
By  the sixth and the seventh step, we can assume that $t\neq e$
and that $u\neq e.$  We treat only the case where $u\in U^+(\kk).$

 Write $u=\prod_{\alpha \in  {}_k\Phi^+} u_{\alpha}$,
 with $u_\alpha\in U_{(\alpha)}$ (see Proposition~\ref{Prop-UnipotentDirectGen}).
Let $\alpha_0\in  {}_k\Phi^+$  be such that  $u_{\alpha_0}\neq e.$
Choose $\la\in X_*(S)$  such that $r:=\langle \alpha_0, \la\rangle \neq 0$
and a sequence $(x_n)_{n\geq 1}$ in $\kk^*$ such that 
$\la(x_n)\in G(\kk)^+$ and $x_m^r\neq x_n^r$ for all $m\neq n.$
One checks, as in the seventh step, that 
$$
\la(x_m) \cdot u_{\alpha_0}\neq \la(x_n) \cdot u_{\alpha_0} \tout n\neq m.
$$ 
Then 
$$
[g^{-1}, \la(x_m)] =  u^{-1} \prod_{\alpha \in {}_k\Phi^+ }\la(x_m)\cdot u_{\alpha}\neq u^{-1}\prod_{\alpha \in  {}_k\Phi^+}\la(x_n)\cdot u_{\alpha}= [g^{-1}, \la(x_n)];
$$
hence, by the sixth step,
$$
\vfi([g^{-1},\la(x_m)]^{-1}[g^{-1}, \la(x_n)] ) =0 \tout m\neq n
$$
and the claim follows.

\vskip.2cm
$\bullet$ {\it Ninth step.} Let $g\in G(\kk)\setminus Z_G(S)$.
We claim that  $\vfi(g)=0$.

We follow the strategy used  in the fourth step.
Recall that $P^\pm = Z_G(S) U^\pm.$
Since  $Z_G(S)= P^+ \cap P^-,$ 
we have either $g\notin P^+(\kk)$ or $g\notin P^-(\kk).$
Assume that, say, $g\notin P^+(\kk).$
It follows from the Bruhat decomposition
(Theorem~\ref{Theo-BruhatDec}) that $g\in U^+(\kk) w P^+(\kk)$
for some $w \in {}_{\kk}W $ with $w\neq e.$
So, the conjugate of $g$ by some element 
 from $U^+(\kk)\subset G(\kk)^+$ belongs to $w P^+(\kk)$.
We can therefore assume that $g$ is of the form 
$$
g= w\ga u \,\, \text{ for some } \ga\in Z_G(S)(\kk) \text{ and } u\in U^+(\kk).
$$

Since $w$ is non trivial, there exists $\la\in X_*(S)$ such that $\la^{w^{-1}}\neq \la.$
So, $\la':= \la^{w^{-1}}\la^{-1}$ is a non trivial one-parameter subgroup 
in $S$  and  we can therefore find
a sequence $(x_n)_{n\geq 1}$ in $\kk^*$ such that 
$\la(x_n)\in G(\kk)^+$ and such that
$$\la'(x_m)^{-1} \la'(x_n)\notin Z(G)(\kk)  \tout m\neq n.$$
We have 
$$
[w^{-1} , \la(x_n)]= \la'(x_n) \tout n\geq 1,
$$
and, setting $u_n:= \la(x_n)\cdot u \in U^+(\kk)$, 
 $$
 [g^{-1}, \la(x_n)]= u^{-1}  \la'(x_n) u_n \tout n\geq 1.
 $$
By the eighth step, it follows that
 $$
 \vfi([g^{-1}, \la(x_m)]^{-1}[g^{-1}, \la(x_n)])= \vfi(\la'(x_m)^{-1} \la'(x_n) (u_n u_m^{-1}))=0
 $$
 for all $m\neq n$ and the claim follows.
  
\vskip.2cm
$\bullet$ {\it Tenth step.} Let $g\in Z_G(S)$ with $g\notin Z(G)(\kk)$.
We claim that $\vfi(g)=0.$

Indeed, in view of the ninth step, it suffices to show
that $G(\kk)\setminus Z_G(S)$ contains a conjugate of $g$
under $G(\kk)^+.$
Assume, by contradiction, that every $G(\kk)^+$-conjugate of
$g$ is contained in $Z_G(S).$ 
Then the subgroup $L$ generated by the $G(\kk)^+$-conjugates of
$g$ is contained in   $Z_G(S).$
Since $g\notin Z(G)(\kk)$,  Tits' simplicity theorem
implies  that $L$ contains $G(\kk)^+$. Hence, $G(\kk)^+$ is contained in 
$Z_G(S)$. This is a contradiction, since  $S$ is non trivial.

\subsection{Proof of the second statement}
The first statement of Theorem~\ref{Theo-CharSemisimple}
implies that every $\vfi\in \Char(G(\kk))$  either factorizes
through $G(\kk)/G(\kk)^+$ or is of the form $\widetilde{\chi}$
for some $\chi \in \widehat{Z(G)(\kk)}.$

Conversely, let $\chi \in \widehat{Z(G)(\kk)}.$
We claim that $\widetilde{\chi}$, which clearly is a $G(\kk)$-invariant
function from $P(G(\kk)),$  is indecomposable.
Let $\vfi_1, \vfi_2$ be two $G(\kk)$-invariant
functions from $P(G(\kk))$ and $t\in (0,1)$ such that
$\widetilde{\chi}= t\vfi_1+(1-t)\vfi_2.$

On the one hand,  by the first statement of 
Theorem~\ref{Theo-CharSemisimple}, each $\vfi_i$
is a convex combination of normalized functions of positive type 
$\psi_i^{(1)}, \psi_i^{(2)}$ with $\psi_i^{(1)}=1$ on 
$G(\kk)^+$ and $\psi_i^{(2)}=0$ outside $Z(G)(\kk).$
Since $\widetilde{\chi}=0$ outside $Z(G)(\kk),$
this implies that $\vfi_i=0$ outside $Z(G)(\kk)$ for $i=1,2.$

On the other hand, since $\widetilde{\chi}|_{Z(G)(\kk)}=\chi \in \Char(Z(G)(\kk)),$
the functions $\vfi_1$ and $\vfi_2$ coincide with $\chi$
on $Z(G)(\kk).$ Hence $\vfi_1=\vfi_2=\widetilde{\chi}.$

\section{Extension of the results and proofs of corollaries}
\label{ExtensionReductive}
\subsection{Extension to semi-simple groups}
  We extend Corollary~\ref{Cor-CharSemisimple} to semi-simple groups.
  First recall the following  general result 
 (see \cite[Satz]{Thoma2}). Let  $\Ga=\Ga_1\times \dots  \times\Ga_n$ be a direct product  of discrete groups.
 Then 
 $$\Char(\Ga)=\left\{ \vfi_1\otimes \cdots \otimes \vfi_n\mid \vfi_1\in \Char(\Ga_1),\dots \vfi_n\in \Char(\Ga_n)\right\},$$
 where $(\vfi_1\otimes \cdots \otimes\vfi_n)(\ga_1,\dots, \ga_2)= \vfi_1(\ga_1)\cdots \vfi_n(\ga_n)$.

Let  $G$ be a connected semi-simple algebraic group defined 
over a field $\kk.$ 
As in the case of simple groups,  define $G(\kk)^+$  as  the  subgroup of $G(\kk)$ generated by the unipotent radicals of parabolic subgroups of $G$ defined over $\kk.$ 

Let $G_1, \dots, G_r$  be the connected almost $k$-simple  and $k$-isotropic normal   $k$-subgroups of $G$. Then $G(\kk)^+$
is the almost direct product of the  $G_i(\kk)^+$'s, 
that is, the product map 
$$p:G_1^+(\kk)\times \cdots \times G_r(\kk)^+\to G(\kk)^+$$
is a surjective homomorphism with finite kernel (see \cite[6.2]{Borel-Tits2}).
In particular,  we may identify $\Char(G(\kk)^+)$
with a subset of characters of $G_1^+(\kk)\times \cdots \times G_r(\kk)^+$.

The following result is an immediate consequence of these remarks.
\begin{proposition}
\label{Theo-Extension}
With $\kk, G, G_1,\dots ,G_r,$ and $p$  as above, 
$\Char(G(\kk)^+)$
 coincides with the set functions $\vfi$ 
of the form $\vfi=\vfi_1 \otimes \cdots \otimes\vfi_r$,
with $\vfi_i \in \Char(G_i(\kk)^+)$ for every $1\leq i\leq r$ and such that $\vfi=1$ on $\ker p.$
\end{proposition} 

\subsection{Proof of Corollary~\ref{Cor-CharSemisimple}}
Let $\vfi\in \Char(G(\kk)^+).$ Then $\widetilde{\vfi}\in P(G(\kk))$ and 
$\widetilde{\vfi}$ is $G(\kk)^+$-invariant. Hence, by Theorem~\ref{Theo-CharSemisimple},
$\widetilde{\vfi}$ is a convex combination of  two functions $\vfi_1$ and $\vfi_2$
from $P(G(\kk))$ with $\vfi_1=1$ on $G(\kk)^+$ and $\vfi_2=0$ outside $Z(G)(\kk).$
Since $\vfi$ is indecomposable, it follows
that either $\vfi=\Un_{G(\kk)^+}$ or $\vfi=\widetilde{\chi}$ 
for some $\chi$ in the dual of $Z(G) \cap G(\kk)^+.$
Since, as observed in Remark~\ref{Rem-Cor-CharSemisimple2},
 $Z(G) \cap G(\kk)^+= Z(G(\kk)^+),$ this completes the proof.

\subsection{Proof of Corollary~\ref{Cor-CharSemisimple2}} 
(i)  Observe first that $PG(\kk)^+$ is a so-called ICC group,
that is, every conjugacy class, except $\{e\}$, is infinite.
Indeed, by Tits simplicity theorem, $PG(\kk)^+$ has no 
proper subgroup of finite index.
This implies that $L(PG(\kk)^+)$ is a factor (see Lemma 5.3.4 in \cite{Murray-VN}).
Using Corollary~\ref{Cor-CharSemisimple}, the claim is proved  as in  \cite[Theorem 3.5]{Peterson-Thom}.

\noindent
(ii) Observe that $PG(\kk)^+$ is countable, since $\kk$ is assumed to be
countable. As $PG(\kk)^+$ is an ICC group, the claim follows
from Corollary~\ref{Cor-CharSemisimple} and   \cite[Theorem 3.2]{Peterson-Thom}
(or \cite[Theorem 2.12]{Dudko-Medynets}).

\subsection{Proof of Corollary~\ref{Cor-Theo-CharSemisimple3}} 
Item (i) is an immediate consequence of  Corollary~\ref{Cor-CharSemisimple} and Remark~\ref{Rem-Cor-CharSemisimple2}.iii.

To show Item (ii), let $\kk$ be a global field and let $G(\kk)\curvearrowright (X,\mu)$ be a measure preserving ergodic action on a probability space $(X,\mu).$
Let $X'$ be the set of points $x\in X$ with finite $G(\kk)^+$-orbit.
By Tits' simplicity theorem, $X'$ coincides with the $G(\kk)^+$-fixed points in $X.$
Since $X'$ is $G(\kk)$-invariant, we have either $\mu(X')=1$ or $\mu(X')=0,$
by ergodicity of  $G(\kk)\curvearrowright (X,\mu)$. In the first case, the action 
of $G(\kk)$ factorizes through $G(\kk)/ G(\kk)^+= G(\kk)_{\rm ab}.$
In the second case, the $G(\kk)$-orbit of $\mu$-almost every point in $X$ is infinite.
Since $G(\kk)$ is a countable ICC-group, it follows from
Item (i) and the proof of  \cite[Theorem 3.2 ]{Peterson-Thom} that $G(\kk)\curvearrowright (X,\mu)$ is essentially free. 

\end{document}